\documentclass[11pt,reqno]{amsart}

\usepackage[T1]{fontenc}
\usepackage{lmodern}
\usepackage{microtype}
\usepackage{amsmath,amssymb,amsthm,mathtools}
\usepackage{xcolor}
\usepackage[hidelinks]{hyperref}
\usepackage{tikz-cd}
\usepackage[nameinlink,noabbrev]{cleveref}
\allowdisplaybreaks
\numberwithin{equation}{section}

\newtheorem{theorem}{Theorem}[section]
\newtheorem{proposition}[theorem]{Proposition}
\newtheorem{lemma}[theorem]{Lemma}
\newtheorem{corollary}[theorem]{Corollary}
\theoremstyle{definition}
\newtheorem{definition}[theorem]{Definition}
\newtheorem{example}[theorem]{Example}
\theoremstyle{remark}
\newtheorem{remark}[theorem]{Remark}

\crefname{theorem}{Theorem}{Theorems}
\crefname{proposition}{Proposition}{Propositions}
\crefname{lemma}{Lemma}{Lemmas}
\crefname{corollary}{Corollary}{Corollaries}
\crefname{remark}{Remark}{Remarks}
\crefname{definition}{Definition}{Definitions}
\crefname{example}{Example}{Examples}

\newtheorem{mainthm}{Theorem}

\crefname{mainthm}{Theorem}{Theorems}
\newtheorem{conjecture}[theorem]{Conjecture}
\crefname{conjecture}{Conjecture}{Conjectures}

\newcommand{\CC}{\mathbb C}
\newcommand{\OO}{\mathcal O}
\newcommand{\m}{\mathfrak m}
\newcommand{\n}{\mathfrak n}
\newcommand{\Der}{\operatorname{Der}}
\newcommand{\Soc}{\operatorname{Soc}}
\newcommand{\mult}{\operatorname{mult}}
\newcommand{\corank}{\operatorname{corank}}
\newcommand{\Hess}{\operatorname{Hess}}
\newcommand{\Spec}{\operatorname{Spec}}
\newcommand{\im}{\operatorname{im}}
\newcommand{\coker}{\operatorname{coker}}
\newcommand{\ol}[1]{\overline{#1}}

\title[Derivations of Generalized Moduli Algebras]
{Derivations of Generalized Moduli Algebras of Isolated Hypersurface
Singularities}

\author{Zhiwen Liu}
\address{Beijing Institute of Mathematical Sciences and Applications (BIMSA), Beijing, 100084, China}
\email{liuzhiwen@bimsa.cn}

\author{Stephen S.-T. Yau}
\address{Department of Mathematical Sciences, Tsinghua University, Beijing, 100084,  China.}
\email{yau@uic.edu}
\date{}

\RequirePackage{color}\definecolor{RED}{rgb}{1,0,0}\definecolor{BLUE}{rgb}{0,0,1} 
\DeclareOldFontCommand{\sf}{\normalfont\sffamily}{\mathsf} 
\providecommand{\DIFaddtex}[1]{{\protect\color{black} #1}} 
\providecommand{\DIFaddbegin}{} 
\providecommand{\DIFaddend}{} 
\providecommand{\DIFadd}[1]{\texorpdfstring{\DIFaddtex{#1}}{#1}} 
\RequirePackage{listings} 
\RequirePackage{color} 
\lstdefinelanguage{DIFcode}{ 
  moredelim=[il][\color{red}\scriptsize]{\%DIF\ <\ }, 
  moredelim=[il][\color{blue}\sffamily]{\%DIF\ >\ } 
} 
\lstdefinestyle{DIFverbatimstyle}{ 
	language=DIFcode, 
	basicstyle=\ttfamily, 
	columns=fullflexible, 
	keepspaces=true 
} 
\lstnewenvironment{DIFverbatim}{\lstset{style=DIFverbatimstyle}}{} 
\lstnewenvironment{DIFverbatim*}{\lstset{style=DIFverbatimstyle,showspaces=true}}{} 

\begin{document}

\begin{abstract}
Let $(V,0)$ be an isolated complex hypersurface singularity defined by
$f\in\CC\{x_1,\ldots,x_n\}$.
Let $A(V)$ be the moduli(Tjurina) algebra,
let $A^*(V)$ be the generalized moduli algebra, and let
\[
 L(V)=\Der_{\CC}(A(V),A(V)),\qquad
 L^*(V)=\Der_{\CC}(A^*(V),A^*(V))
\]
be the Yau algebra and the new Yau algebra of \(V\), respectively.
Write $\lambda(V)=\dim_{\CC}L(V)$ and
$\lambda^*(V)=\dim_{\CC}L^*(V)$.
We prove that the difference between these two dimensions is determined by
the Hessian corank.  In the Morse case, the generalized moduli algebra is zero, and hence
both derivation Lie algebras are zero. There is 
\[
 \lambda^*(V)=\lambda(V)-
 \begin{cases}
  1,&\corank\Hess(f)(0)=1,\\
  0,&\corank\Hess(f)(0)\ne1.
 \end{cases}
\]
In particular, $\lambda^*(V)=\lambda(V)$ when $n\ge2$ and
$\mult(f)\ge3$.  This proves Conjecture~1.1 proposed in \cite{ChenHussainYauZuo2020}. 
We construct an exact sequence whose two end terms are
copies of the socle of the Milnor algebra. This homological approach reveals the detailed algebraic structure underlying the dimension count.  We also give a alternative proof that pushes the original arguments in the proof of \cite[Theorem~C]{ChenHussainYauZuo2020}  further.
  The general case is obtained by combining Saito's criterion for the
non-quasi-homogeneous case.
\end{abstract}

\maketitle
\enlargethispage{2pt}

\section{Introduction}

Let
\[
 \OO_n=\CC\{x_1,\ldots,x_n\},
 \qquad
 \m=(x_1,\ldots,x_n),
\]
and let $f\in\m^2$ have an isolated critical point at the origin.  Set
$(V,0)=V(f)$ and write
\[
 J_f=\left(
 \frac{\partial f}{\partial x_1},\ldots,
 \frac{\partial f}{\partial x_n}
 \right),
 \qquad
 h_f=\det\Hess(f),
\]
{\color{black}where $J_f$ is the Jacobian ideal and $h_f$ is the Hessian determinant.}
The quotient
\[
 A(V)=\OO_n/(f,J_f)
\]
is the \emph{moduli algebra}, also known as the Tjurina algebra.  The
Mather--Yau theorem states that its isomorphism class determines and is determined by the
analytic isomorphism type of the isolated hypersurface germ
\cite{MatherYau1982}.

A key object of study is the derivation Lie algebra of $A(V)$, 
\[
 L(V)=\Der_{\CC}(A(V),A(V)),
\]
which is a finite-dimensional Lie algebra.
  Yau proved that for any isolated hypersurface singularity, $L(V)$ is a solvable Lie algebra; see \cite{Yau1986},\cite{2}. This  result led to $L(V)$ being named the Yau algebra \cite{Yau1986,4,5}, and
\[
 \lambda(V)=\dim_{\CC}L(V)
\]
is its \emph{Yau number} \cite{ElashviliKhimshiashvili2006}. This establishes a profound connection between singularity theory and solvable Lie algebras.

In \cite{ChenHussainYauZuo2020},the authors introduced the  Artinian local algebra, i.e., the finite-dimensional quotient
algebra
\[
 A^*(V)=\OO_n/(f,J_f,h_f)
\]
and called it the \emph{generalized moduli algebra}
\cite{ChenHussainYauZuo2020}.  {\color{black}The construction is motivated by Dimca's
theorem on zero-dimensional isolated complete intersections
\cite{Dimca1984}: in the quasi-homogeneous case, the singular subspace of
the Jacobian complete intersection has coordinate algebra $A^*(V)$.
Together with the Mather--Yau theorem, this shows that $A^*(V)$ is a
complete invariant within the quasi-homogeneous class; see
\cite[Remark~2.1]{ChenHussainYauZuo2020}. } The corresponding derivation
Lie algebra and its dimension are
\[
 L^*(V)=\Der_{\CC}(A^*(V),A^*(V)),
 \qquad
 \lambda^*(V)=\dim_{\CC}L^*(V).
\]
We refer to $L^*(V)$ as the \emph{new Yau algebra}.  When $A^*(V)$ is
the zero algebra, we use the convention $\Der_{\CC}(A^*(V),A^*(V))=0$.

{\color{black}
The construction of Yau algebras belongs to a broader interaction
between singularity theory and finite-dimensional Lie theory.
Brieskorn's work relating rational double points to simple algebraic
groups provides a classical instance of this interaction
\cite{Brieskorn1970}. 
By the Levi decomposition, every finite-dimensional complex Lie algebra
admits a decomposition
\[
 \mathfrak g=\mathfrak s\ltimes\operatorname{rad}(\mathfrak g),
\]
where $\mathfrak s$ is a semisimple Levi subalgebra and
$\operatorname{rad}(\mathfrak g)$ is the solvable radical of
$\mathfrak g$. The semisimple part admits a uniform classification in terms of root
systems and Dynkin diagrams \cite{Humphreys1972,Jacobson1962}, whereas solvable, and in particular nilpotent, Lie algebras do not admit
a comparably uniform classification \cite{Malcev1950}.

  Yau's
construction provides a complementary geometric direction: it
associates with every isolated hypersurface singularity a
finite-dimensional solvable Lie algebra.  It therefore gives a
geometrically organized family of solvable Lie algebras that can be
studied through local invariants of singularities; see also
\cite{Yau1986,Yau1991,SeeleyYau1990}.

The new Yau algebra belongs to the same solvable Lie-theoretic framework
throughout the range $\mult(f)\ge4$.  Indeed, if $f$ is not
quasi-homogeneous, then $A^*(V)=A(V)$ and consequently
$L^*(V)=L(V)$.  If $f$ is quasi-homogeneous and $\mult(f)\ge4$, it has been proved that $L^*(V)$ is solvable
\cite[Corollary~2.1]{ChenHussainYauZuo2020}.  Hence, for every isolated
hypersurface singularity of multiplicity at least four, $L(V)$ and
$L^*(V)$ are two finite-dimensional solvable Lie algebras naturally
associated with the same singularity.  

The comparison of their dimensions also has a direct geometric
interpretation in terms of local algebras.  For a finite-dimensional
commutative $\CC$-algebra $B$, the derivation algebra
$\Der_{\CC}(B,B)$ is the tangent Lie algebra at the identity of the
algebraic automorphism group of $B$; equivalently, it describes the
infinitesimal algebra automorphisms of $B$.  Thus $\lambda(V)$ and
$\lambda^*(V)$ measure the dimensions of the infinitesimal symmetry
algebras encoded by the two intrinsic local algebras $A(V)$ and
$A^*(V)$, respectively.  Since
\[
 A^*(V)=A(V)/(\overline{h_f}),
\]
determining the difference $\lambda(V)-\lambda^*(V)$ asks how imposing
the additional Hessian relation changes the dimension of the
infinitesimal symmetry algebra. 

From both the Lie-theoretic and geometric
viewpoints, it is therefore a natural and meaningful numerical problem.
The following conjecture was proposed in
\cite{ChenHussainYauZuo2020}.
}

\begin{conjecture}[{\cite[Conjecture~1.1]{ChenHussainYauZuo2020}}]
\label{conj:CHYZ}
Let $n\ge2$, and let $f\in\OO_n$ define an isolated hypersurface
singularity of multiplicity at least three.  Then
\begin{equation}\label{eq:CHYZ-conjecture-intro}
 \lambda^*(V)=\lambda(V).
\end{equation}
\end{conjecture}

The non-quasi-homogeneous case was already settled in the same paper.
Indeed, Saito's criterion gives $h_f\in(f,J_f)$, and hence
$A^*(V)=A(V)$ \cite[Corollary~3.8]{Saito1974}.  

For weighted
homogeneous singularities, the authors proved
$\lambda^*(V)\le\lambda(V)$ in arbitrary dimension and equality when
$2\le n\le4$.  They also obtained equality for canonical weights
\[
 d\ge2w_1\ge2w_2\ge\cdots\ge2w_n>0
\]
under the additional condition $w_n\ge w_1/2$, and for binomial
singularities; see
\cite[Theorems~C and~D, Remark~6.1, and Theorem~7.1]{ChenHussainYauZuo2020}.

Our first result proves \cref{conj:CHYZ} for weighted homogeneous
singularities in every dimension, without an additional restriction on
the weights.  Combining this with the non-quasi-homogeneous case proves
the conjecture in full.  We then remove the multiplicity assumption and distinguish  the two dimensions using Hessian corank.

\begin{mainthm}[Main theorem]\label{thm:main-intro}
Let $f\in\m^2\subset\OO_n$ define an isolated hypersurface singularity,
and set
\[
c(f):=\corank\Hess(f)(0)
=\dim_{\CC}\ker\Hess(f)(0),
\]
where
\[
\Hess(f)(0)
=\left(\frac{\partial^2f}{\partial x_i\partial x_j}(0)\right)_{1\le i,j\le n}.
\]
Then the following statements hold.
\begin{enumerate}
\item If $c(f)=0$, then $f$ is Morse,
\[
 A(V)\cong\CC,
 \qquad
 A^*(V)=0,
\]
and
\[
 \lambda(V)=\lambda^*(V)=0.
\]
\item If $c(f)=1$, then the Milnor number $\mu=\mu(V)$ satisfies
$\mu\ge2$, and
\[
 A(V)\cong\CC[t]/(t^\mu),
 \qquad
 A^*(V)\cong\CC[t]/(t^{\mu-1}).
\]
Hence
\[
 \lambda(V)=\mu-1,
 \qquad
 \lambda^*(V)=\mu-2.
\]
\item If $c(f)\ge2$, then
\[
 \lambda^*(V)=\lambda(V).
\]
\end{enumerate}
\end{mainthm}

We indicate the main steps of the proof.  Suppose first that $f$ is a
weighted homogeneous polynomial, $n\ge2$, and $\mult(f)\ge3$.  The
moduli algebra is then the Milnor algebra
\[
 M_f=\CC[x_1,\ldots,x_n]/J_f,
\]
and the class of $h_f$ generates its one-dimensional socle.  The
adjugate-Hessian derivations and the relevant degree estimate appear
already in the proof of \cite[Theorem~C]{ChenHussainYauZuo2020}.  Using
relative derivations, we organize this calculation into the exact
sequence of complex vector spaces
\begin{equation}\label{eq:intro-hessian-cramer}
\begin{aligned}
 0\longrightarrow\Soc(M_f)^{\oplus n}
 &\longrightarrow\Der_{\CC}(M_f,M_f)\\
 &\longrightarrow\Der_{\CC}(A^*(V),A^*(V))
 \longrightarrow\Soc(M_f)^{\oplus n}
 \longrightarrow0.
\end{aligned}
\end{equation}
The two end terms have the same dimension, giving the required dimension
equality.  We also record a shorter proof that combines the calculation
in \cite[Theorem~C]{ChenHussainYauZuo2020} with preservation of the
socle under derivations.

Finally, the holomorphic splitting lemma
\cite[Theorem~2.47]{GreuelLossenShustin2007} reduces an arbitrary
isolated hypersurface germ to a residual germ in $c(f)$ variables.  The
Morse case has no residual variables, one residual variable gives the
corank-one exception, and a residual germ in at least two variables has
multiplicity at least three.  This proves the full corank formula.

The paper is organized as follows.   Section~2  \DIFaddbegin \DIFadd{develops the local
analytic background, the finite-dimensional algebras attached to a
hypersurface germ, their }\DIFaddend derivation Lie algebras, and  \DIFaddbegin \DIFadd{the required
invariance facts}\DIFaddend.  In Section~3 we
review socles and treat the non-quasi-homogeneous case.  In Section~4 we
give the two proofs described above.  In Section~5 we use the splitting
lemma to prove
\cref{thm:main-intro}.  Section~6 gives consequences and examples.

\section{Preliminaries and invariance}

\subsection{\DIFaddbegin \DIFadd{Hypersurface germs and their local algebras}\DIFaddend}

Let
 \DIFaddbegin \[
 \DIFadd{P_n=\CC[x_1,\ldots,x_n]
}\]
\DIFadd{be the polynomial ring in $n$ variables.  The local analytic counterpart
of $P_n$ is the ring $\OO_n$ of germs of holomorphic functions at the
origin of $\CC^n$.  Recall that a germ records a holomorphic function only
in an unspecified neighborhood of the origin: two functions represent the
same germ when they agree on some smaller neighborhood.  Taking Taylor
series identifies
}\[
 \DIFadd{\OO_n=\CC\{x_1,\ldots,x_n\},
}\]
\DIFadd{the ring of convergent power series.  It is a local $\CC$-algebra with
maximal ideal
}\[
 \DIFadd{\m=(x_1,\ldots,x_n)
   =\{g\in\OO_n:g(0)=0\}
}\]
\DIFadd{and residue field $\OO_n/\m\cong\CC$.  A germ $u\in\OO_n$ is a unit
precisely when $u(0)\neq0$.
}

\DIFadd{For $f\in\m$, the zero set of a representative of $f$ in a sufficiently
small neighborhood of the origin defines a hypersurface germ, denoted by
}\[
 \DIFadd{(V,0)=V(f)\subset(\CC^n,0).
}\]
\DIFadd{Only the behavior near the origin is retained.  In particular, replacing
$f$ by $uf$, where $u$ is a unit, does not change the hypersurface germ.
The local function algebra of $V$ is
}\[
 \DIFadd{F(V)=\OO_n/(f),
}\]
\DIFadd{where $(f)$ is the principal ideal generated by $f$.  This algebra is
usually infinite-dimensional over $\CC$, but it is the natural local
coordinate ring of the germ.
}

\DIFadd{Throughout the paper, $f\in\m^2$, so the origin is a critical point.
The multiplicity of $f$ is
}\[
 \DIFadd{\mult(f)=\max\{r\ge1:f\in\m^r\};
}\]
\DIFadd{equivalently, it is the least total degree of a nonzero term in the
power-series expansion of $f$.  Write
}\[
 \DIFadd{J_f=\left(
 \frac{\partial f}{\partial x_1},\ldots,
 \frac{\partial f}{\partial x_n}
 \right)
}\]
\DIFadd{for the Jacobian ideal.  We say that $f$ has an }\DIFaddend isolated critical point
at the origin  \DIFaddbegin \DIFadd{if, in a sufficiently small neighborhood, the gradient
$\nabla f$ vanishes only at the origin.  Algebraically, this is equivalent
to
}\[
 \DIFadd{\sqrt{J_f}=\m,
}\]
\DIFadd{or, equivalently, to $J_f$ being }\DIFaddend $\m$-primary\DIFaddbegin \DIFadd{.  Thus the quotient
}\[
 \DIFadd{M_f:=\OO_n/J_f
}\]\DIFaddend 
is finite-dimensional.   \DIFaddbegin \DIFadd{It is called the \emph{Milnor algebra} of the
defining germ $f$.  Whenever the defining germ is fixed, we also use the
notation $M(V)=M_f$.  Its dimension
}\[
 \DIFadd{\mu(V)=\dim_{\CC}M_f
}\]\DIFaddend 
is the Milnor number.
\DIFaddbegin 

\DIFadd{The Hessian matrix and its determinant are denoted by
}\[
 \DIFadd{\Hess(f)=
 \left(\frac{\partial^2f}{\partial x_i\partial x_j}\right)_{1\le i,j\le n},
 \qquad
 h_f=\det\Hess(f).
}\]
\DIFaddend If $\ol f$ and $\ol h_f$  \DIFaddbegin \DIFadd{are their classes }\DIFaddend in $M(V)$, then  the moduli
algebra  \DIFaddbegin \DIFadd{can be written as
}\[
 \DIFadd{A(V)=\OO_n/(f,J_f)=M(V)/(\ol f).
}\]
Thus $A(V)$ is its local
coordinate algebra( also called the \emph{Tjurina algebra}).  The  generalized moduli algebra \DIFaddbegin \DIFadd{is
}\[
 \DIFadd{A^*(V)=\OO_n/(f,J_f,h_f)
       =M(V)/(\ol f,\ol h_f),
}\]
\DIFadd{and is obtained by supposing the additional Hessian relation.  Both }\DIFaddend are
finite-dimensional  \DIFaddbegin \DIFadd{quotients of the Milnor
algebra.  The number
}\DIFaddend \[
 \tau(V)=\dim_{\CC}A(V)
\]
is the Tjurina number.
The  \DIFaddbegin \DIFadd{Mather--Yau theorem shows that the finite-dimensional algebra $A(V)$
 determines and is determined by the analytic isomorphism type of the isolated
hypersurface germ }\cite{MatherYau1982}\DIFadd{.  The role of $A^*(V)$ and the
reason for adjoining the Hessian determinant are recalled below.
}\DIFaddend 

 \DIFaddbegin \subsection{\DIFadd{Weighted homogeneous singularities and the Hessian socle}}

\DIFadd{Assign positive rational weights $w_1,\ldots,w_n$ to the variables.  The
weighted degree of a monomial is
}\[
 \DIFadd{\deg_w(x_1^{a_1}\cdots x_n^{a_n})
 =\sum_{i=1}^n a_iw_i.
}\]
\DIFadd{A }\DIFaddend polynomial $f$ is \emph{weighted homogeneous}, also called
\emph{quasi-homogeneous},  \DIFaddbegin \DIFadd{of }\DIFaddend weighted degree $d$  \DIFaddbegin \DIFadd{if every monomial
}\DIFaddend occurring in $f$  \DIFaddbegin \DIFadd{has weighted degree $d$.  The tuple
}\[
 \DIFadd{(w_1,\ldots,w_n;d)
}\]\DIFaddend 
 \DIFaddbegin \DIFadd{is its weight type.  Multiplying all }\DIFaddend weights and $d$  \DIFaddbegin \DIFadd{by a common positive
integer allows them to be taken integral.  A }\DIFaddend hypersurface germ is called
quasi-homogeneous if\DIFaddbegin \DIFadd{, }\DIFaddend after an analytic change of coordinates\DIFaddbegin \DIFadd{, it has a
weighted homogeneous representative}\DIFaddend.
\DIFaddbegin 

\DIFadd{The weighted Euler vector field
}\[
 \DIFadd{E=\sum_{i=1}^n w_i x_i\frac{\partial}{\partial x_i}
}\]
\DIFadd{satisfies the Euler identity
}\[
 \DIFadd{E(f)=\sum_{i=1}^n w_i x_i\frac{\partial f}{\partial x_i}=df.
}\]
\DIFaddend Consequently,  \DIFaddbegin \DIFadd{$f\in J_f$ for a weighted homogeneous representative, and
therefore
}\[
 \DIFadd{A(V)=M(V),\qquad \tau(V)=\mu(V).
}\]\DIFaddend 
 \DIFaddbegin \DIFadd{The grading also gives
}\[
 \DIFadd{\deg_w\left(\frac{\partial f}{\partial x_i}\right)=d-w_i,
 \qquad
 \deg_w\left(\frac{\partial^2f}{\partial x_i\partial x_j}\right)
 =d-w_i-w_j.
}\]
\DIFadd{Every term in the Hessian determinant consequently has weighted degree
}\[
 \DIFadd{\deg_w(h_f)=nd-2\sum_{i=1}^n w_i.
}\]
\DIFadd{As explained below, this is the socle degree denoted by $\sigma$ in
Section~4.
}

\DIFadd{For an isolated weighted homogeneous singularity, the Milnor number may
also be read from the weights:
}\begin{equation}\DIFadd{\label{eq:weighted-milnor-number}
 \mu(V)=\prod_{i=1}^n\frac{d-w_i}{w_i};
}\end{equation}
\DIFadd{indeed, the partial derivatives form a weighted homogeneous regular
sequence of degrees $d-w_1,\ldots,d-w_n$, and the formula follows from
the corresponding complete-intersection Hilbert series; see
}\cite[Proposition~2.1]{ChenHussainYauZuo2020}\DIFadd{.  When the weights are
ordered and satisfy
}\[
 \DIFadd{d\ge2w_1\ge2w_2\ge\cdots\ge2w_n>0,
}\]
\DIFadd{the weight type is called canonical. 
}

\DIFadd{We next recall the socle statement that links the Hessian determinant to
the generalized moduli algebra.  If $(R,\n)$ is a finite-dimensional
local $\CC$-algebra, its socle is
}\[
 \DIFadd{\Soc(R)=(0:_R\n)
        =\{r\in R:\n r=0\}.
}\]
\DIFadd{Thus the socle consists of the elements annihilated by the maximal ideal;
if $\n^{s+1}=0$ and $\n^s\neq0$, then
$\n^s\subseteq\Soc(R)$.  
Since $J_f$ is $\m$-primary and is generated by
$n$ elements in the regular local ring $\OO_n$, its generators form a
regular sequence.  Hence $M(V)$ is an
Artinian complete intersection and, in particular, an Artinian
Gorenstein algebra; see }\cite[Chapter~2]{BrunsHerzog1993}\DIFadd{.  A theorem of
Scheja and Storch identifies its socle generator
}\cite{SchejaStorch1975}\DIFadd{:
}\begin{equation}\DIFadd{\label{eq:milnor-socle}
 \Soc(M(V))=\CC\ol{h_f}.
}\end{equation}
\DIFadd{In the }\DIFaddend quasi-homogeneous case\DIFaddbegin \DIFadd{, $f\in J_f$, and it follows that
}\[
 \DIFadd{A^*(V)=M(V)/(\ol h_f)=M(V)/\Soc(M(V)).
}\]\DIFaddend 
\DIFaddbegin \DIFadd{Consequently,
}\[
 \DIFadd{\dim_{\CC}A^*(V)=\mu(V)-1.
}\]
\DIFaddend When $\mu(V)=1$, the  \DIFaddbegin \DIFadd{Milnor algebra is $\CC$, its socle is the whole
algebra, and $A^*(V)$ is the }\DIFaddend zero algebra.

\subsection{\DIFaddbegin \DIFadd{Derivations and the Yau }\DIFaddend algebras}

Let $R$ be a commutative $\CC$-algebra and let $N$ be an $R$-module.  A
$\CC$-linear map $D:R\to N$ is a derivation if \DIFaddbegin \DIFadd{it satisfies the Leibniz
rule
}\DIFaddend \[
 D(ab)=aD(b)+bD(a)
\]
for all $a,b\in R$.   \DIFaddbegin \DIFadd{In particular, $D(1)=0$.  We write
$\Der_{\CC}(R,N)$ for }\DIFaddend the vector space of such maps\DIFaddbegin \DIFadd{.  }\DIFaddend Equivalently,
\[
 \Der_{\CC}(R,N)\cong
 \operatorname{Hom}_R(\Omega_{R/\CC},N),
\]
where $\Omega_{R/\CC}$ is the module of K\"ahler differentials; see
\cite[Section~16.1]{Eisenbud1995}.

 \DIFaddbegin \DIFadd{For a concrete presentation, let $P=\CC[x_1,\ldots,x_n]$ and
$R=P/I$.  A derivation $D:R\to N$ is determined by the elements
}\[
 \DIFadd{u_j=D(\ol{x_j})\in N,\qquad 1\le j\le n.
}\]
\DIFadd{For every relation $g\in I$, these elements must satisfy
}\[
 \DIFadd{0=D(\ol g)
   =\sum_{j=1}^n
     \ol{\frac{\partial g}{\partial x_j}}\,u_j.
}\]
\DIFadd{Conversely, any tuple $(u_1,\ldots,u_n)$ satisfying these relations
defines a derivation.  In particular, when $I=J_f$, differentiating the
relations $\partial f/\partial x_i$ produces the Hessian matrix.  This is
the presentation of relative derivations used in Section~4.
}

\DIFadd{When $N=R$, the commutator
}\[
 [\DIFadd{D_1,D_2}]\DIFadd{=D_1D_2-D_2D_1
}\]
\DIFadd{is again a derivation.  It therefore makes $\Der_{\CC}(R,R)$ a Lie
algebra.  If $R$ is }\DIFaddend finite-dimensional\DIFaddbegin \DIFadd{, then so is this Lie algebra.
Moreover, an algebra isomorphism $\alpha:R\to S$ induces a Lie algebra
isomorphism
}\[
 \DIFadd{\Der_{\CC}(R,R)\longrightarrow\Der_{\CC}(S,S),
 \qquad
 D\longmapsto\alpha D\alpha^{-1}.
}\]\DIFaddend 
 \DIFaddbegin 

\DIFadd{For the moduli algebra, this construction gives the Yau algebra
}\[
 \DIFadd{L(V)=\Der_{\CC}(A(V),A(V)),
 \qquad
 \lambda(V)=\dim_{\CC}L(V).
}\]
\DIFadd{Yau proved that $L(V)$ is solvable for every isolated hypersurface
singularity }\cite{Yau1986,Yau1991}\DIFadd{.  The corresponding objects for the
generalized moduli algebra are the new Yau algebra and its dimension:
}\[
 \DIFadd{L^*(V)=\Der_{\CC}(A^*(V),A^*(V)),
 \qquad
 \lambda^*(V)=\dim_{\CC}L^*(V).
}\]
\DIFaddbegin \DIFadd{When
$A^*(V)$ is the zero algebra, we use the convention
}\DIFaddend $\Der_{\CC}(0,0)=0$.

If $I\subset R$ is an ideal and $\pi:R\to R/I$ is the quotient
map, composition with $\pi$ identifies
\[
 \Der_{\CC}(R/I,R/I)
 \cong
 \{D\in\Der_{\CC}(R,R/I):D(I)=0\}.
\]
\DIFadd{Thus a derivation from $R$ to $R/I$ descends to the quotient precisely
when it annihilates the defining ideal.
}\DIFaddend 

\DIFaddbegin \subsection{\DIFadd{The generalized moduli algebra and contact invariance}}

\DIFadd{The definition of $A^*(V)$ has a natural complete-intersection
interpretation in the quasi-homogeneous case.  The partial derivatives of
$f$ define the zero-dimensional complete intersection
}\[
 \DIFadd{X_f=\Spec M(V).
}\]
\DIFadd{Its Jacobian matrix is the Hessian matrix of $f$.  Hence the singular
subspace of $X_f$ is defined by the Jacobian ideal together with the
Hessian determinant, and its coordinate algebra is
}\[
 \DIFadd{\OO_n/(J_f,h_f)=A^*(V),
}\]
\DIFadd{where the last equality uses $f\in J_f$.  Dimca's theorem states that a
zero-dimensional isolated complete intersection is determined by its
singular subspace }\cite{Dimca1984}\DIFadd{.  Together with the Mather--Yau
theorem, this gives the analytic invariance for $A^*(V)$ within the
quasi-homogeneous class recorded in
}\cite[Remark~2.1]{ChenHussainYauZuo2020}\DIFadd{.
}

\DIFaddend Two germs $f,g\in\OO_n$ are \emph{right equivalent} if
$g=f\circ\phi$ for an analytic coordinate change
$\phi:(\CC^n,0)\to(\CC^n,0)$.  They are \emph{contact equivalent} if
\[
 g=u\,(f\circ\phi)
\]
for some unit $u\in\OO_n^\times$. 

For $1\le k\le n$, let $\mathfrak h_k(g)$ be the ideal generated by the
$k\times k$ minors of $\Hess(g)$ and set
\[
 H_k(g)=\OO_n/\bigl((g)+J_g+\mathfrak h_k(g)\bigr).
\]
Dimca and Sticlaru proved that the isomorphism class of $H_k(g)$ is
invariant under contact equivalence
\cite[Lemma~2.1]{DimcaSticlaru2015}.  Since
$\mathfrak h_n(g)=(h_g)$, the generalized moduli algebra is the top
local Hessian algebra:
\[
 A^*(V(g))=H_n(g).
\]
We record this case explicitly because the transformation formula also
proves the invariance of the Hessian corank used in
\cref{thm:main-intro}.
\begin{proposition}\label{prop:contact-invariance}
Let $f,g\in\m^2\subset\OO_n$ define isolated hypersurface singularities.
Assume that
\[
 g=u\,(f\circ\phi),
\]
where $u\in\OO_n^\times$ and
$\phi:(\CC^n,0)\to(\CC^n,0)$ is an analytic coordinate change.  Then
\[
 A(V(f))\cong A(V(g)),\qquad
 A^*(V(f))\cong A^*(V(g)).
\]
Moreover,
\[
 \corank\Hess(f)(0)=\corank\Hess(g)(0).
\]
\end{proposition}

\begin{proof}
We first consider $g=f\circ\phi$.  Let $J_\phi$ be the Jacobian matrix of
$\phi$.  The chain rule gives
\[
 \nabla g=J_\phi^{t}(\nabla f)\circ\phi.
\]
Since $J_\phi$ is invertible over $\OO_n$, pullback by $\phi$ identifies the
Jacobian ideals.  A second differentiation gives
\begin{equation}\label{eq:hessian-coordinate}
 \Hess(g)=J_\phi^{t}(\Hess(f)\circ\phi)J_\phi
 +\sum_{k=1}^n(f_k\circ\phi)\Hess(\phi_k),
\end{equation}
where $f_k=\partial f/\partial x_k$.  The entries of the second term lie in
$J_g$.  Taking determinants modulo $J_g$ gives
\[
 h_g\equiv(\det J_\phi)^2(h_f\circ\phi)\pmod{J_g}.
\]
The factor $(\det J_\phi)^2$ is a unit.  This proves the two algebra
isomorphisms under a coordinate change.  Since $\nabla f(0)=0$,
evaluation at the origin gives
\[
 \Hess(g)(0)=J_\phi(0)^t\Hess(f)(0)J_\phi(0),
\]
so the corank is unchanged.

Now let $g=uf$ with $u$ a unit.  The equality
\[
 \nabla g=u\nabla f+f\nabla u
\]
implies
\[
 (g,J_g)=(f,J_f).
\]
Also,
\begin{equation}\label{eq:hessian-unit}
 \Hess(g)=u\Hess(f)+f\Hess(u)
 + (\nabla u)(\nabla f)^t+(\nabla f)(\nabla u)^t.
\end{equation}
Modulo $(f,J_f)$, this gives
\[
 h_g\equiv u^n h_f.
\]
Since $u^n$ is a unit, the generalized moduli algebras are isomorphic.  At
the origin,
\[
 \Hess(g)(0)=u(0)\Hess(f)(0),
\]
so the corank is again unchanged.
\end{proof}

\DIFaddbegin \DIFadd{
The splitting lemma shows
that the Hessian corank is exactly the number of residual variables left
after the nondegenerate quadratic variables have been removed.  This is
why the statement of \cref{thm:main-intro} is naturally organized by
Hessian corank.
}


\section{Socles and the non-quasi-homogeneous case}

Let $(R,\n)$ be a finite-dimensional local $\CC$-algebra.  The next two
facts about its socle are standard.  We include their short proofs because
they are used at the central step of the argument.  General background on
Artinian Gorenstein rings and socles can be found in
\cite[Chapter~3]{BrunsHerzog1993}.

\begin{lemma}\label{lem:nonzero-ideal-meets-socle}
Every nonzero ideal of $R$ meets $\Soc(R)$ nontrivially.  If $R$ is
Artinian Gorenstein, then every nonzero ideal contains $\Soc(R)$.
\end{lemma}
\begin{proof}
Let $I \subseteq R$ be a nonzero ideal. Since $R$ is an Artinian local ring, its maximal ideal $\mathfrak{m}$ is nilpotent. Thus, there exists a well-defined maximum integer $r \ge 0$ such that $\mathfrak{m}^r I \neq 0$. By the maximality of $r$, we have $\mathfrak{m}^{r+1} I = 0$, which means $\mathfrak{m}(\mathfrak{m}^r I) = 0$. Consequently, the nonzero submodule $\mathfrak{m}^r I$ is annihilated by $\mathfrak{m}$, implying that $0 \neq \mathfrak{m}^r I \subseteq \operatorname{Soc}(R)$. This proves the first assertion.

If $R$ is additionally Gorenstein, $\operatorname{Soc}(R)$ is a one-dimensional vector space over the residue field $R/\mathfrak{m}$. Since $\mathfrak{m}^r I$ is a nonzero subspace of $\operatorname{Soc}(R)$, it must coincide with the entire socle. Hence, $\operatorname{Soc}(R) = \mathfrak{m}^r I \subseteq I$, completing the proof.
\end{proof}

\begin{lemma}[Socle preservation]\label{lem:socle-preservation}
Let $(R,\n)$ be a finite-dimensional local $\CC$-algebra.  Every
$D\in\Der_{\CC}(R,R)$ satisfies
\[
 D(\n)\subseteq\n,
 \qquad
 D(\Soc(R))\subseteq\Soc(R).
\]
\end{lemma}
\begin{proof}
This is the local commutative case of Hochschild's theorem that the radical of a finite-dimensional algebra over a field of characteristic zero is a characteristic ideal \cite[Theorem 4.2]{Hochschild1942}. We give a direct, self-contained proof.

Let us first show that $D(\mathfrak{n}) \subseteq \mathfrak{n}$. Since $R$ is a finite-dimensional local $\mathbb{C}$-algebra, it is an Artinian local ring. Therefore, its unique maximal ideal $\mathfrak{n}$ coincides with its nilradical, which implies that every element $a \in \mathfrak{n}$ is nilpotent.

 Let $a\in\n$, and suppose that $D(a)$
is a unit.  Choose the least $r\ge1$ with $a^r=0$.  Since the
characteristic is zero,
\[
 0=D(a^r)=r a^{r-1}D(a).
\]
Cancelling the unit $D(a)$ gives $a^{r-1}=0$, a contradiction.  Thus
$D(a)\in\n$.

Now let $s\in\Soc(R)$ and $a\in\n$.  Since $as=0$,
\[
 aD(s)=D(as)-D(a)s=0.
\]
The last term is zero because $D(a)\in\n$.  Hence $D(s)$ is annihilated by
$\n$ and lies in the socle.
\end{proof}

\begin{proposition}[The non-quasi-homogeneous case]\label{prop:non-qh}
If $f$ is not quasi-homogeneous, then
\[
 A^*(V)=A(V).
\]
Consequently,
\[
 L^*(V)=L(V),\qquad \lambda^*(V)=\lambda(V).
\]
\end{proposition}

\begin{proof}
This conclusion was already recorded in
\cite[p.~441]{ChenHussainYauZuo2020}, using
\cite[Corollary~3.8]{Saito1974}.  We give the equivalent socle argument.

By Saito's criterion, $f\in J_f$ if and only if $(V,0)$ is
quasi-homogeneous \cite{Saito1971}.  Thus
$\ol f\ne0$ in $M(V)$ in the present case.  By
\cref{lem:nonzero-ideal-meets-socle}, the nonzero ideal $(\ol f)$
contains the one-dimensional socle of $M(V)$.  In view of
\eqref{eq:milnor-socle},
\[
 \ol h_f\in(\ol f).
\]
Therefore $(\ol f,\ol h_f)=(\ol f)$ in $M(V)$, and the two quotient
algebras are equal.
\end{proof}

\section{The weighted homogeneous case}
Throughout this section, let
$P=\CC[x_1,\ldots,x_n]$, where $n\ge2$, and let $f\in P$ be weighted
homogeneous of degree $d$ with positive integer weights
$w_1,\ldots,w_n$.  Assume that the origin is an isolated critical point
of $f$ and that $\mult(f)\ge3$.  The Jacobian ideal $J_f$ is then
$\m$-primary, so $P/J_f$ agrees naturally with the analytic Milnor
algebra $\OO_n/J_f\OO_n$.  Set
\[
 M=P/J_f,\qquad
 H=\Hess(f),\qquad
 h=\det H,\qquad
 B=M/(\ol h).
\]
Euler's identity and \eqref{eq:milnor-socle} identify these algebras as
\[
 M=A(V),\qquad B=A^*(V)=M/\Soc(M).
\]

The partial derivatives of $f$ form a homogeneous regular sequence of
degrees $d-w_1,\ldots,d-w_n$.  Set
\[
 W=\sum_{i=1}^n w_i,
 \qquad
 \sigma=nd-2W.
\]
The standard Hilbert-series formula for a graded complete intersection is
\begin{equation}\label{eq:hilbert-series-M}
\operatorname{Hilb}_M(t)=
 \prod_{i=1}^n\frac{1-t^{d-w_i}}{1-t^{w_i}}.
\end{equation}
It follows that
\begin{equation}\label{eq:top-degree}
 M_q=0\quad(q>\sigma),
 \qquad
 M_\sigma=\CC\ol h,
 \qquad
 B_q=0\quad(q\ge\sigma).
\end{equation}

The adjugate-Hessian derivations and the degree estimate below occur
already in the proof of \cite[Theorem~C]{ChenHussainYauZuo2020}.  Our
purpose is to place that calculation in an exact sequence and then
analyze the obstruction to descending a relative derivation to $B$.

\subsection{Relative derivations and the snake lemma}

If $N$ is an $M$-module, let $\Der_{\CC}(M,N)$ denote the space of
$\CC$-derivations from $M$ to $N$.

\begin{lemma}\label{lem:relative-kernel}
There are natural identifications
\[
 \Der_{\CC}(M,M)\cong\ker(H:M^n\to M^n),
\]
and
\[
 \Der_{\CC}(M,B)\cong\ker(H:B^n\to B^n).
\]
\end{lemma}

\begin{proof}
A derivation $D:M\to N$ is determined by
\[
 u_j=D(\ol{x_j}),\qquad 1\le j\le n.
\]
The relations $f_i=\partial f/\partial x_i=0$ give
\[
 0=D(\ol{f_i})=\sum_{j=1}^n\ol{f_{ij}}u_j.
\]
Thus the column vector $(u_1,\ldots,u_n)^t$ lies in the kernel of the
Hessian matrix.  Conversely, any vector in this kernel defines a
derivation on the quotient.
\end{proof}

\subsection{The Cramer derivations}

We now compare the two kernels
\[
 \ker(H:M^n\to M^n)
 \qquad\text{and}\qquad
 \ker(H:B^n\to B^n),
\]
where
\[
 H=\operatorname{Hess}(f)
\]
is viewed as a matrix with entries in \(M\), and also as a matrix with
entries in \(B\) after reduction modulo \((\overline h)\).

By the result of Scheja and Storch \cite{SchejaStorch1975},
\[
 \Soc(M)=\CC\ol h.
\]

Set
\[
 \Sigma=\Soc(M)^{\oplus n}
        =(\CC\ol h)^{\oplus n}
        \subset M^n.
\]

\begin{proposition}
\label{prop:relative-cramer}
There is an exact sequence of finite-dimensional complex vector spaces
\begin{equation}\label{eq:relative-cramer}
 0\longrightarrow\Sigma
 \longrightarrow\Der_{\CC}(M,M)
 \longrightarrow\Der_{\CC}(M,B)
 \xrightarrow{\partial}\Sigma
 \longrightarrow0.
\end{equation}
In particular,
\[
 \dim_{\CC}\Der_{\CC}(M,B)=\dim_{\CC}\Der_{\CC}(M,M).
\]

Moreover, for \(p=1,\ldots,n\), let \(D_p\in\Der_{\CC}(M,B)\) be the
relative derivation defined by
\begin{equation}\label{eq:Cramer-derivation}
 D_p(\ol{x_j})
 =
 \pi\bigl(\ol{\operatorname{adj}(H)_{jp}}\bigr),
 \qquad
 1\le j\le n,
\end{equation}
where \(\pi:M\twoheadrightarrow B\) is the quotient map. Then, with the
usual convention for the connecting map,
\begin{equation}\label{eq:boundary-Cramer}
 \partial(D_p)=\ol h\,e_p,
\end{equation}
where \(e_p\) is the \(p\)-th standard basis vector of \(M^n\). Hence the
classes of
\[
 D_1,\ldots,D_n
\]
span the quotient of \(\Der_{\CC}(M,B)\) by the image of
\(\Der_{\CC}(M,M)\).
\end{proposition}

\begin{proof}
We use the Hessian matrix
\[
 H=\operatorname{Hess}(f)
\]
as a matrix over \(M\), and also as a matrix over \(B\) after reducing
modulo \((\ol h)\).
By \cref{lem:relative-kernel},
\[
 \Der_{\CC}(M,M)
 \cong
 \ker(H:M^n\to M^n)
\]
and
\[
 \Der_{\CC}(M,B)
 \cong
 \ker(H:B^n\to B^n).
\]
Under these identifications, the map
\[
 \Der_{\CC}(M,M)\longrightarrow \Der_{\CC}(M,B)
\]
is induced by the quotient map
\[
 \pi^n:M^n\longrightarrow B^n.
\]

Since
\[
 B=M/(\ol h)
\]
and
\[
 \Soc(M)=\CC\ol h,
\]
the kernel of \(\pi^n\) is exactly
\[
 \Sigma=(\CC\ol h)^{\oplus n}.
\]
Thus we have a short exact sequence
\[
 0\longrightarrow
 \Sigma
 \longrightarrow
 M^n
 \xrightarrow{\pi^n}
 B^n
 \longrightarrow0.
\]

The multiplicity assumption is used at this point.  Since
$\mult(f)\ge3$, every second partial derivative
\[
 f_{ij}=\frac{\partial^2 f}{\partial x_i\partial x_j}
\]
has no constant term.  Hence every entry of \(H\) lies in the maximal
ideal of \(M\).  The maximal ideal kills the socle, so
\[
 H\Sigma=0.
\]
Therefore \(H\) gives a commutative diagram with exact rows:
\[
\begin{tikzcd}
0 \arrow[r] &
\Sigma \arrow[r] \arrow[d,"0"'] &
M^n \arrow[r,"\pi^n"] \arrow[d,"H"] &
B^n \arrow[r] \arrow[d,"H"] &
0
\\
0 \arrow[r] &
\Sigma \arrow[r] &
M^n \arrow[r,"\pi^n"] &
B^n \arrow[r] &
0.
\end{tikzcd}
\]
The left vertical arrow is zero precisely because \(H\Sigma=0\).

Applying the snake lemma to this diagram gives an exact sequence
\[
\begin{aligned}
0\longrightarrow
\Sigma
&\longrightarrow
\ker(H:M^n\to M^n)
\longrightarrow
\ker(H:B^n\to B^n)
\\
&\xrightarrow{\partial}
\Sigma
\longrightarrow
\operatorname{coker}(H:M^n\to M^n)
\longrightarrow
\operatorname{coker}(H:B^n\to B^n).
\end{aligned}
\]
For completeness, the connecting map is as follows.  Take
\[
 \ol v\in\ker(H:B^n\to B^n)
\]
and choose a lift
\[
 v\in M^n.
\]
Since \(H\ol v=0\) in \(B^n\), we have
\[
 \pi^n(Hv)=0.
\]
Thus
\[
 Hv\in\ker(\pi^n)=\Sigma.
\]
The connecting map is
\[
 \partial(\ol v)=Hv.
\]
This is independent of the choice of lift.  Indeed, if \(v'\) is another
lift, then
\[
 v'-v\in\Sigma,
\]
and hence
\[
 H(v'-v)=0
\]
because \(H\Sigma=0\).

It remains to prove that this connecting map is surjective.

Cramer's rule gives
\[
 H\operatorname{adj}(H)=hI_n.
\]
Passing to \(M\) gives
\[
 H\operatorname{adj}(H)=\ol h\,I_n.
\]
Therefore, for the \(p\)-th standard basis vector \(e_p\in M^n\),
\[
 H\bigl(\operatorname{adj}(H)e_p\bigr)=\ol h\,e_p.
\]
The vectors
\[
 \ol h\,e_1,\ldots,\ol h\,e_n
\]
form a basis of \(\Sigma\). Hence
\[
 \Sigma\subseteq\operatorname{im}(H:M^n\to M^n).
\]
Equivalently, the map
\[
 \Sigma\longrightarrow
 \operatorname{coker}(H:M^n\to M^n)
\]
in the snake-lemma sequence is zero. Thus
\[
 \partial:
 \ker(H:B^n\to B^n)
 \longrightarrow
 \Sigma
\]
is surjective. The long exact sequence therefore truncates to
\[
0\longrightarrow
\Sigma
\longrightarrow
\ker(H:M^n\to M^n)
\longrightarrow
\ker(H:B^n\to B^n)
\xrightarrow{\partial}
\Sigma
\longrightarrow0.
\]
Using Lemma \ref{lem:relative-kernel} gives
\[
0\longrightarrow
\Sigma
\longrightarrow
\Der_{\CC}(M,M)
\longrightarrow
\Der_{\CC}(M,B)
\xrightarrow{\partial}
\Sigma
\longrightarrow0.
\]
Taking dimensions gives
\[
 \dim_{\CC}\Der_{\CC}(M,B)
 =
 \dim_{\CC}\Der_{\CC}(M,M).
\]

For \(p=1,\ldots,n\), let
\[
 c^{(p)}=\operatorname{adj}(H)e_p
 =
 \bigl(\operatorname{adj}(H)_{1p},\ldots,
       \operatorname{adj}(H)_{np}\bigr)^t.
\]
Define \(D_p\in\Der_{\CC}(M,B)\) by
\[
 D_p(\ol{x_j})
 =
 \pi\bigl(\ol{c^{(p)}_j}\bigr)
 =
 \pi\bigl(\ol{\operatorname{adj}(H)_{jp}}\bigr),
 \qquad j=1,\ldots,n.
\]

This is a well-defined relative derivation because
\[
 H\,\pi(\operatorname{adj}(H)e_p)
 =
 \pi(H\operatorname{adj}(H)e_p)
 =
 \pi(\ol h\,e_p)
 =0
\]
in \(B^n\).  If we lift \(D_p\) to the vector
\[
 \operatorname{adj}(H)e_p\in M^n,
\]
then the formula for the connecting map gives
\[
 \partial(D_p)
 =
 H\operatorname{adj}(H)e_p
 =
 \ol h\,e_p.
\]
Hence \(D_1,\ldots,D_n\) map to a basis of \(\Sigma\).  Their classes
therefore span the quotient of \(\Der_{\CC}(M,B)\) by the image of
\(\Der_{\CC}(M,M)\).
\end{proof}

\begin{definition}[Cramer derivations]
The relative derivations
\[
 D_1,\ldots,D_n\in\Der_{\CC}(M,B)
\]
constructed in Proposition \ref{prop:relative-cramer} are called the
\emph{Cramer derivations}, or \emph{adjugate-Hessian derivations}.
They are obtained from the columns of $\operatorname{adj}(H)$ and
satisfy
\[
 D_p(\ol{x_j})
 =
 \pi\bigl(\ol{\operatorname{adj}(H)_{jp}}\bigr),
 \qquad
 j=1,\ldots,n,
\]
and
\[
 \partial(D_p)=\ol h\,e_p.
\]
\end{definition}

\subsection{The descent obstruction}

A relative derivation $D:M\to B$ descends to a derivation of $B$ if and
only if it kills the kernel $(\ol h)$ of $M\twoheadrightarrow B$.  Define
\[
 \operatorname{ob}:\Der_{\CC}(M,B)\longrightarrow B,
 \qquad
 \operatorname{ob}(D)=D(\ol h).
\]

\begin{lemma}\label{lem:obstruction}
The image of $\operatorname{ob}$ is contained in $\Soc(B)$, and
\begin{equation}\label{eq:obstruction-exact}
 0\longrightarrow\Der_{\CC}(B,B)
 \longrightarrow\Der_{\CC}(M,B)
 \xrightarrow{\operatorname{ob}}\Soc(B)
\end{equation}
is exact.
\end{lemma}

\begin{proof}
Let $\ol a$ lie in the maximal ideal of $M$.  Since
$\ol a\,\ol h=0$ in $M$, the Leibniz rule gives
\[
 0=D(\ol a\,\ol h)
 =\pi(\ol a)D(\ol h)+\pi(\ol h)D(\ol a)
 =\pi(\ol a)D(\ol h),
\]
where $D\in\Der_{\CC}(M,B)$.
Thus $D(\ol h)$ is annihilated by the maximal ideal of $B$ and lies in
$\Soc(B)$.  The kernel of $\operatorname{ob}$ consists exactly of relative
derivations that annihilate $(\ol h)$, hence exactly of derivations of $B$.
\end{proof}

Note that there is a natural map $\phi: \Der_{\CC}(M,M) \to \Der_{\CC}(M, B)$ given by the composition $\tilde{D} \mapsto \pi \circ \tilde{D}$.
By Lemma \ref{lem:socle-preservation}, every derivation of $M$ sends $\ol h$ to
a scalar multiple of $\ol h$.  Therefore  we have
\[
 \operatorname{ob}(\phi(\tilde{D})) = (\pi \circ \tilde{D})(\ol h) = \pi(c\ol h) = c\pi(\ol h) = 0,
\]
since $\ol h \in \ker(\pi)$.

Hence, evaluating the obstruction map on the entire image of $\phi$ identically yields
\begin{equation}\label{eq:ob-image-M-zero}
 \operatorname{ob}\bigl(\im\bigl(\phi: \Der_{\CC}(M,M) \to \Der_{\CC}(M,B)\bigr)\bigr) = 0.
\end{equation}
It remains to evaluate the obstruction on the Cramer derivations.

\subsection{The Cramer degrees}

Each entry $\operatorname{adj}(H)_{jp}$ of the adjugate Hessian matrix
is weighted homogeneous of degree
\[
 \sigma-d+w_j+w_p.
\]
Therefore $D_p$ is homogeneous of degree
\begin{equation}\label{eq:Cramer-degree}
 \epsilon_p=\sigma-d+w_p=(n-1)d-2W+w_p.
\end{equation}
This degree formula is the one used in the proof of
\cite[Theorem~C]{ChenHussainYauZuo2020}.

\begin{lemma}[Nonnegativity of the Cramer degrees]
\label{lem:Cramer-nonnegative}
For every $p=1,\ldots,n$, one has $\epsilon_p\ge0$.
\end{lemma}

\begin{proof}
This is the weight estimate in
\cite{ChenHussainYauZuo2020}; we repeat it for completeness.
Choose $r$ such that
\[
 w_r=\max\{w_1,\ldots,w_n\}.
\]
Since the critical point is isolated, $f$ contains a monomial of the form
$x_r^a x_j$ for some $a\ge1$ and some $j$.  Otherwise every first partial
derivative would vanish on the $x_r$-axis.  The assumption
$\mult(f)\ge3$ gives $a\ge2$.  Hence
\[
 d=aw_r+w_j\ge2w_r+w_j.
\]
For any $p$,
\begin{align*}
 \epsilon_p
 &=(n-1)d-2W+w_p\\
 &\ge2(n-1)w_r+(n-1)w_j-2W+w_p\\
 &\ge2(n-1)w_r+w_j+w_p-2W\\
 &=2(n-1)w_r-(2W-w_j-w_p).
\end{align*}
The last parenthesis is a sum of $2n-2$ weights, counted with
multiplicity.  Each is at most $w_r$.  Thus
\[
 2W-w_j-w_p\le2(n-1)w_r,
\]
and $\epsilon_p\ge0$.
\end{proof}

Since $\ol h$ has degree $\sigma$, equations
\eqref{eq:Cramer-degree} and \eqref{eq:top-degree} give

\begin{equation}\label{eq:obstruction-cramer-zero}
 \operatorname{ob}(D_p)=D_p(\ol h)
 \in B_{\sigma+\epsilon_p}=0.
\end{equation}
Together with the preceding exact sequence, this yields the following
consequence.

\begin{proposition}\label{prop:obstruction-vanishes}
The obstruction map
\[
 \operatorname{ob}:\Der_{\CC}(M,B)\longrightarrow \Soc(B)
\]
is the zero map.
\end{proposition}

\begin{proof}
Let
\[
 \phi:\Der_{\CC}(M,M)\longrightarrow \Der_{\CC}(M,B),
 \qquad
 \widetilde D\longmapsto \pi\circ \widetilde D
\]
be the natural map.

By \eqref{eq:ob-image-M-zero},
\[
 \operatorname{ob}(\im\phi)=0.
\]

By Proposition \ref{prop:relative-cramer}, the classes of the adjugate-Hessian
derivations
\[
 D_1,\ldots,D_n
\]
span the quotient
\[
 \Der_{\CC}(M,B)/\im\phi.
\]

By \eqref{eq:obstruction-cramer-zero}, one gets
\[
 \operatorname{ob}(D_p)=0
\]
for every $p$.  These classes span the quotient, so
$\operatorname{ob}=0$ on all of $\Der_{\CC}(M,B)$.
\end{proof}

\begin{theorem}
\label{thm:hessian-cramer}
Let \(f\in\CC[x_1,\ldots,x_n]\), \(n\ge2\), be weighted homogeneous
with positive weights. Assume that \(f\) has an isolated critical point at
the origin and \(\mult(f)\ge3\). Let
\[
 M=\CC[x_1,\ldots,x_n]/J_f,
 \qquad
 B=M/(\ol h_f).
\]
Then there is an exact sequence of finite-dimensional complex vector spaces
\begin{equation}\label{eq:hessian-cramer-final}
 0\longrightarrow\Soc(M)^{\oplus n}
 \longrightarrow\Der_{\CC}(M,M)
 \longrightarrow\Der_{\CC}(B,B)
 \longrightarrow\Soc(M)^{\oplus n}
 \longrightarrow0.
\end{equation}
Consequently,
\[
 \dim_{\CC}\Der_{\CC}(B,B)
 =
 \dim_{\CC}\Der_{\CC}(M,M).
\]
\end{theorem}

\begin{proof}
By Lemma \ref{lem:obstruction}, the natural map
\[
 \Der_{\CC}(B,B)\longrightarrow \Der_{\CC}(M,B),
 \qquad
 \delta\longmapsto \delta\circ\pi,
\]
identifies \(\Der_{\CC}(B,B)\) with
\[
 \ker(\operatorname{ob}).
\]
By Proposition \ref{prop:obstruction-vanishes}, $\operatorname{ob}=0$.  Hence
\[
 \Der_{\CC}(B,B)\cong \Der_{\CC}(M,B).
\]
Substituting this isomorphism into \eqref{eq:relative-cramer} gives
\eqref{eq:hessian-cramer-final}.  Taking dimensions yields
\begin{equation}\label{eq:dimension-equality}
 \dim_{\CC}\Der_{\CC}(B,B)
 =
 \dim_{\CC}\Der_{\CC}(M,M).
\end{equation}
\end{proof}
\begin{remark}\label{rem:not-Lie-exact}
The sequence \eqref{eq:hessian-cramer-final} is an exact sequence of vector
spaces.  The middle map is a homomorphism of Lie algebras, but the entire
sequence is not asserted to be exact in a category of Lie algebras.
\end{remark}

\subsection{An alternative proof}

The previous proof of theorem \ref{thm:hessian-cramer} using the snake lemma provides a self-contained, exact-sequence reformulation of the adjugate calculation in \cite{ChenHussainYauZuo2020}. We emphasize this homological approach because it reveals the complete algebraic structure underlying the dimension count. Specifically, the  exact sequence \eqref{eq:hessian-cramer-final} not only captures the dimension equality, but also explicitly explains why both the kernel and cokernel are exactly $n$-dimensional, identifies the concrete derivation representatives for the cokernel, determines their precise grading degrees, and clarifies exactly how the $\mult(f) \ge 3$ assumption geometrically eliminates the descent obstruction to the socle quotient.

We now present a second, alternative proof that pushes the original arguments in the proof of \cite[Theorem~C]{ChenHussainYauZuo2020}  further. While it bypasses the rich structural information provided by the exact sequence, we record this direct approach.

\begin{proposition}
\label{prop:short-CHYZ-proof}
Under the assumptions of \cref{thm:hessian-cramer}, one has
\[
 \lambda^*(V)=\lambda(V).
\]
\end{proposition}

\begin{proof}
Use the notation
\[
 I=J_f,\qquad I'=I+(h_f).
\]
The proof of \cite[Theorem~C]{ChenHussainYauZuo2020} considers the natural
map
\begin{equation}\label{eq:CHYZ-natural-map}
 \varphi:
 \Der_{I'/I}(P/I)\longrightarrow\Der_{\CC}(P/I').
\end{equation}
Here $\Der_{I'/I}(P/I)$ denotes the derivations of $P/I$ that preserve
the ideal $I'/I$.  The kernel and cokernel calculations in that proof
give
\[
 \dim_{\CC}\ker\varphi=n,
 \qquad
 \dim_{\CC}\coker\varphi=n.
\]
The ideal $I'/I$ is the one-dimensional socle of $P/I$.  By
Lemma \ref{lem:socle-preservation}, every derivation of $P/I$ preserves this
ideal.  Hence
\[
 \Der_{I'/I}(P/I)=\Der_{\CC}(P/I).
\]
Taking dimensions in \eqref{eq:CHYZ-natural-map} gives
\[
 \dim_{\CC}\Der_{\CC}(P/I')
 =\dim_{\CC}\Der_{\CC}(P/I).
\]
This is the desired equality.
\end{proof}

\begin{theorem}[Conjecture~1.1 in \cite{ChenHussainYauZuo2020}]
\label{thm:CHYZ-conjecture}
Let $n\ge2$, and let $f\in\OO_n$ define an isolated hypersurface
singularity with $\mult(f)\ge3$.  Then
\[
 \lambda^*(V)=\lambda(V).
\]
\end{theorem}

\begin{proof}
If $(V(f),0)$ is not quasi-homogeneous, apply Proposition \ref{prop:non-qh}.  In
the quasi-homogeneous case, choose a weighted homogeneous polynomial
representative after an analytic change of coordinates.  By
Proposition \ref{prop:contact-invariance}, the moduli algebras, generalized moduli
algebras, and their derivation Lie algebras are unchanged.  The result
now follows from
\cref{thm:hessian-cramer}.
\end{proof}

\section{Splitting and the Hessian corank}

We now pass from the multiplicity statement to the full corank
classification.  The holomorphic splitting lemma, also called the
generalized Morse lemma, states that a germ $f\in\m^2$ of Hessian
corank $c$ is right equivalent to
\begin{equation}\label{eq:splitting-form}
 f\stackrel{r}{\sim}
 g(x_1,\ldots,x_c)+y_1^2+\cdots+y_{n-c}^2,
\end{equation}
where $g\in(x_1,\ldots,x_c)^3$ has an isolated critical point when
$c>0$ and is unique up to right equivalence
\cite[Theorem~2.47]{GreuelLossenShustin2007}.  We first record the
effect of the quadratic summands on the algebras in question.

\begin{proposition}\label{prop:suspension}
Let $g\in\CC\{x_1,\ldots,x_c\}$ have an isolated critical point, and set
\[
 F(x,y)=g(x)+y_1^2+\cdots+y_r^2.
\]
Then
\[
 A(V(F))\cong A(V(g)),
 \qquad
 A^*(V(F))\cong A^*(V(g)).
\]
Consequently,
\[
 L(V(F))\cong L(V(g)),
 \qquad
 L^*(V(F))\cong L^*(V(g))
\]
as Lie algebras.
\end{proposition}

\begin{proof}
The Jacobian ideal of $F$ is
\[
 J_F=(J_g,y_1,\ldots,y_r).
\]
Thus all $y_i$ vanish in the two quotient algebras.  The Hessian matrix is
block diagonal:
\[
 \Hess(F)=
 \begin{pmatrix}
  \Hess(g)&0\\
  0&2I_r
 \end{pmatrix}.
\]
Hence
\[
 h_F=2^r h_g.
\]
The algebra isomorphisms follow, and derivation Lie algebras are functorial
under algebra isomorphisms.
\end{proof}

The standard classification of germs of corank at most one as type
$A_k$ is stated in
\cite[Theorem~2.48]{GreuelLossenShustin2007}.  We include the short
one-variable calculation because it also identifies the two Artinian
algebras used here.

\begin{lemma}[One residual variable]\label{lem:one-variable}
Let $g\in\CC\{t\}$ have an isolated critical point and order at least
three.  Then $g$ is right equivalent to $t^{\mu+1}$ for a unique integer
$\mu\ge2$.  Moreover,
\[
 A(V(g))\cong\CC[t]/(t^\mu),
 \qquad
 A^*(V(g))\cong\CC[t]/(t^{\mu-1}).
\]
\end{lemma}

\begin{proof}
Write
\[
 g(t)=t^{\mu+1}u(t),\qquad u(0)\ne0.
\]
An analytic $(\mu+1)$-st root of the unit $u$ gives a coordinate change
that sends $g$ to $t^{\mu+1}$.  For this normal form,
\[
 g'(t)=(\mu+1)t^\mu,
 \qquad
 g''(t)=\mu(\mu+1)t^{\mu-1}.
\]
The quotient presentations follow.
\end{proof}

\begin{proof}[Proof of \cref{thm:main-intro}]
Apply the splitting lemma and Proposition \ref{prop:contact-invariance}, and then
remove the quadratic variables by Proposition \ref{prop:suspension}.  It is enough
to study the residual germ $g$ in $c=c(f)$ variables.

If $c=0$, the holomorphic Morse lemma gives a nondegenerate quadratic
form.  Its moduli algebra is $\CC$, while its Hessian determinant is a unit.
Thus $A^*(V)=0$, and both derivation spaces are zero.

If $c=1$, apply Lemma \ref{lem:one-variable}.  For $N\ge1$, a derivation of
$\CC[t]/(t^N)$ is determined by an arbitrary element of $(t)$ as the value
of $t$.  Hence
\[
 \dim_{\CC}\Der_{\CC}(\CC[t]/(t^N))=N-1.
\]
This gives
\[
 \lambda(V)=\mu-1,
 \qquad
 \lambda^*(V)=\mu-2.
\]

If $c\ge2$, then the residual germ $g$ has multiplicity at least three in
$c$ variables.  Apply \cref{thm:CHYZ-conjecture}.
\end{proof}

\section{Consequences and examples}

\begin{corollary}\label{cor:original-conjecture}
If $n\ge2$ and $\mult(f)\ge3$, then
\[
 \lambda^*(V)=\lambda(V).
\]
\end{corollary}

\begin{proof}
The Hessian matrix vanishes at the origin, so
$c(f)=n\ge2$.  Apply \cref{thm:main-intro}.
\end{proof}

\begin{corollary}[The corank-one exception]\label{cor:defect}
Define
\[
 \delta_H(V)=\lambda(V)-\lambda^*(V).
\]
Then $\delta_H(V)\in\{0,1\}$, and $\delta_H(V)=1$ if and only if $V$ is
stably right equivalent to an $A_k$ singularity with $k\ge2$.
\end{corollary}

\begin{proof}
The splitting lemma reduces the corank-one case to one variable, where
Lemma \ref{lem:one-variable} gives the normal form $t^{k+1}$.  The converse is
immediate from the same normal form.  The numerical statement follows
from \cref{thm:main-intro}.
\end{proof}

\begin{corollary}[Plane curves and simple singularities]
\label{cor:plane-ADE}
For an isolated plane curve singularity, strict inequality
$\lambda^*(V)<\lambda(V)$ occurs exactly for type $A_k$, $k\ge2$.
Among the simple hypersurface singularities, the same statement holds
after adding or removing nondegenerate quadratic variables.  In
particular,
\[
 \lambda(A_k)=k-1,\qquad \lambda^*(A_k)=k-2,
\]
while $\lambda^*=\lambda$ for $D_k,E_6,E_7,E_8$.
\end{corollary}

\begin{proof}
The standard normal forms are recalled in
\cite{ArnoldGuseinZadeVarchenko1985}.  Type $A_k$ has one residual
variable.  The residual forms of $D_k,E_6,E_7,E_8$ have two variables and
order at least three.  Apply \cref{thm:main-intro}.
\end{proof}

\begin{example}[The Lie algebras for $A_k$]\label{ex:Ak}
Let $k\ge2$ and
\[
 f=t^{k+1}+y_1^2+\cdots+y_r^2.
\]
Then
\[
 A(V)\cong\CC[t]/(t^k),
 \qquad
 A^*(V)\cong\CC[t]/(t^{k-1}).
\]
A basis of $L(V)$ is
\[
 e_i=t^i\frac{d}{dt},\qquad 1\le i\le k-1,
\]
with bracket
\[
 [e_i,e_j]=
 \begin{cases}
  (j-i)e_{i+j-1},&i+j-1\le k-1,\\
  0,&i+j-1\ge k.
 \end{cases}
\]
The natural map $L(V)\to L^*(V)$ is surjective, and its kernel is
$\CC e_{k-1}$.  The one-dimensional defect is visible at the level of Lie
algebras.
\end{example}

\begin{example}[Multiplicity two with equality]\label{ex:mult-two-equality}
Consider
\[
 f=x^3+y^3+z_1^2+\cdots+z_r^2,
 \qquad r\ge1.
\]
The multiplicity is two, but the Hessian corank is two.  By quadratic
suspension,
\[
 A(V)\cong\CC[x,y]/(x^2,y^2),
\]
and
\[
 A^*(V)\cong\CC[x,y]/(x^2,xy,y^2).
\]
The first derivation algebra has basis
\[
 x\partial_x,\quad xy\partial_x,\quad
 y\partial_y,\quad xy\partial_y.
\]
The second algebra has a square-zero maximal ideal of dimension two, so
its derivation algebra is the full endomorphism algebra of that maximal
ideal.  Both dimensions are four.  Thus multiplicity at least three is a
sufficient condition, not a necessary one.
\end{example}

\begin{example}[A multiplicity-two counterexample]\label{ex:counterexample}
Let
\[
 f=x^a y+y^2,
 \qquad a\ge2.
\]
Completing the square gives
\[
 f=\left(y+\frac{1}{2}x^a\right)^2-\frac{1}{4}x^{2a}.
\]
Thus $f$ is a quadratic suspension of a one-variable germ and has Hessian
corank one.  Hence
\[
 A(V)\cong\CC[x]/(x^{2a-1}),
 \qquad
 A^*(V)\cong\CC[x]/(x^{2a-2}),
\]
and
\[
 \lambda(V)-\lambda^*(V)=1.
\]
The failure is caused by corank one, not by multiplicity two alone.
\end{example}

\begin{example}[A family crossing corank strata]\label{ex:family}
Consider
\[
 f_s=x^3+y^3+sxy+z^2,
 \qquad s\in\CC.
\]
For $s\ne0$, the Hessian at the origin is nondegenerate.  Thus $f_s$ is
Morse and
\[
 \lambda(V(f_s))=\lambda^*(V(f_s))=0.
\]
At $s=0$, the Hessian corank is two.  Quadratic suspension gives
\[
 A(V(f_0))\cong\CC[x,y]/(x^2,y^2),
\]
\[
 A^*(V(f_0))\cong\CC[x,y]/(x^2,xy,y^2),
\]
so
\[
 \lambda(V(f_0))=\lambda^*(V(f_0))=4.
\]
The common dimension jumps, but the equality is preserved on both
corank strata.
\end{example}


\appendix

\section{Technical account of AIM-assisted exploratory work}
\label{app:aim}

This appendix records the role of the AI Mathematician (AIM) system
described in \cite{https://arxiv.org/html/2505.22451v1} during the exploratory stage of this
project.  It does not add any hypothesis or result to the main text.
All mathematical statements used in the paper were checked by the
authors.

\subsection{Starting point and responsibilities}

The starting problem was to compare the Yau algebra and the new Yau
algebra
\[
 L(V)=\Der_{\CC}(A(V),A(V)),
 \qquad
 L^*(V)=\Der_{\CC}(A^*(V),A^*(V)),
\]
and, in particular, to determine whether
\[
 \lambda^*(V)=\lambda(V).
\]
In the weighted homogeneous case, write \(M=A(V)\) and \(B=A^*(V)\).
The authors proposed to compare
\(\Der_{\CC}(M,M)\) and \(\Der_{\CC}(B,B)\) through the intermediate
space \(\Der_{\CC}(M,B)\).  The Hessian matrix describes
\(\Der_{\CC}(M,B)\) as a kernel, while the gradient of the Hessian
determinant gives the additional condition for an element of this space
to define a derivation of \(B\).  Exact sequences were then used to
compare the three derivation spaces.  AIM was used to organize possible
steps in this approach, test examples, and examine candidate
intermediate statements.

The AIM output was treated as exploratory material, not as a proof.
The authors were responsible for the final theorem and its proof. They verified every argument included in the paper.  The authors take
responsibility for all statements and proofs.

\subsection{The weighted homogeneous case}

Let \(f\in\CC[x_1,\ldots,x_n]\), \(n\ge2\), be weighted homogeneous with
an isolated critical point, and assume that \(\mult(f)\ge3\).  Set
\[
 M=\CC[x_1,\ldots,x_n]/J_f,
 \qquad
 H=\Hess(f),
 \qquad
 h=\det H,
 \qquad
 B=M/(\ol h).
\]
Euler's identity and the classical theorem in
\cite{SchejaStorch1975} give
\[
 M=A(V),
 \qquad
 \Soc(M)=\CC\ol h,
 \qquad
 B=A^*(V)=M/\Soc(M).
\]
If \(f\) has weighted degree \(d\), the variables have positive weights
\(w_1,\ldots,w_n\), and \(W=\sum_iw_i\), then
\[
 \operatorname{Hilb}_M(t)
 =
 \prod_{i=1}^n\frac{1-t^{d-w_i}}{1-t^{w_i}},
 \qquad
 \sigma=nd-2W.
\]
In particular,
\[
 M_\sigma=\CC\ol h,
 \qquad
 B_q=0\quad(q\ge\sigma),
 \qquad
 \dim_{\CC}B=\mu(V)-1.
\]
These are standard facts and were used as inputs to the exploration.

The
relations defining derivations give
\[
 \Der_{\CC}(M,B)
 \cong
 \ker\bigl(H:B^n\longrightarrow B^n\bigr)
\]
and
\[
 \Der_{\CC}(B,B)
 \cong
 \ker\left(
 \begin{bmatrix}
 H\\
 \ol{\nabla h}^{\,t}
 \end{bmatrix}
 :B^n\longrightarrow B^{n+1}
 \right).
\]
Under these identifications, a derivation \(D:M\to B\) induces a
derivation of \(B=M/(\ol h)\) precisely when it vanishes on
\[
 (\ol h)=\ker(M\longrightarrow B),
\]
or equivalently when \(D(\ol h)=0\).  This condition is measured by
\[
 \operatorname{ob}:
 \ker\bigl(H:B^n\longrightarrow B^n\bigr)
 \longrightarrow B,
 \qquad
 u\longmapsto\ol{\nabla h}^{\,t}u.
\]
Equivalently, if \(u\) represents \(D\in\Der_{\CC}(M,B)\), then
\[
 \operatorname{ob}(D)=D(\ol h).
\]
It follows that
\[
 \dim_{\CC}\Der_{\CC}(M,B)
 -
 \dim_{\CC}\Der_{\CC}(B,B)
 =
 \operatorname{rank}_{\CC}(\operatorname{ob}).
\]
This converts the dimension comparison into a question about one
specified map.


\subsection{Exact sequence and Cramer derivations}
The authors proposed to compare
\(\Der_{\CC}(M,M)\) and \(\Der_{\CC}(B,B)\) through the intermediate
space \(\Der_{\CC}(M,B)\).  A pre-manuscript note generated with AIM
developed this proposal into a preliminary two-step argument for the
weighted homogeneous case.

The adjugate-Hessian derivations and the degree estimate used in this
argument already appear in
\cite[Theorem~C]{ChenHussainYauZuo2020}.  The AIM-generated note did not
originate these ingredients.  It used them in the comparison of the
three derivation spaces described above.

The first step compares \(\Der_{\CC}(M,M)\) with
\(\Der_{\CC}(M,B)\).  Since
\[
 B=M/(\ol h),
 \qquad
 \Soc(M)=\CC\ol h,
\]
the quotient map \(M^n\to B^n\) gives a short exact sequence
\[
 0\longrightarrow\Soc(M)^{\oplus n}
 \longrightarrow M^n
 \longrightarrow B^n
 \longrightarrow0.
\]
The Hessian matrix acts on both \(M^n\) and \(B^n\).  Since
\(\mult(f)\ge3\), its entries lie in the maximal ideal of \(M\), and
hence
\[
 H\Soc(M)^{\oplus n}=0.
\]

Therefore  there is a commutative diagram with exact rows:
\[
\begin{tikzcd}
0 \arrow[r] &
\Soc(M)^{\oplus n}
\arrow[r] \arrow[d,"0"'] &
M^n
\arrow[r] \arrow[d,"H"] &
B^n
\arrow[r] \arrow[d,"H"] &
0
\\
0 \arrow[r] &
\Soc(M)^{\oplus n}
\arrow[r] &
M^n
\arrow[r] &
B^n
\arrow[r] &
0.
\end{tikzcd}
\]
Applying the snake lemma to this diagram, and then using
$
 H\operatorname{adj}(H)=hI_n
$
gives the exact sequence
\[
 0\longrightarrow\Soc(M)^{\oplus n}
 \longrightarrow\Der_{\CC}(M,M)
 \longrightarrow\Der_{\CC}(M,B)
 \xrightarrow{\partial}\Soc(M)^{\oplus n}
 \longrightarrow0.
\]
Consequently,
\[
 \dim_{\CC}\Der_{\CC}(M,B)
 =
 \dim_{\CC}\Der_{\CC}(M,M).
\]

For \(p=1,\ldots,n\), the \(p\)-th column of
\(\operatorname{adj}(H)\) defines a derivation \(D_p:M\to B\) by
\[
 D_p(\ol{x_j})
 =
 \pi\bigl(\ol{\operatorname{adj}(H)_{jp}}\bigr),
 \qquad 1\le j\le n,
\]
where \(\pi:M\to B\) is the quotient map.  Cramer's identity gives
\[
 \partial(D_p)=\ol h\,e_p.
\]
The classes of \(D_1,\ldots,D_n\) therefore span the quotient
\[
 Q:=
 \Der_{\CC}(M,B)\big/
 \im\bigl(
 \Der_{\CC}(M,M)\longrightarrow\Der_{\CC}(M,B)
 \bigr).
\]

The second step determines which derivations \(D:M\to B\) define
derivations of the quotient algebra \(B=M/(\ol h)\).  Such a derivation
defines a derivation of \(B\) precisely when
\[
 D(\ol h)=0.
\]
The AIM-generated note organized this condition through the map
\[
 \operatorname{ob}:\Der_{\CC}(M,B)\longrightarrow B,
 \qquad
 \operatorname{ob}(D)=D(\ol h).
\]
Preservation of the one-dimensional socle of \(M\) shows that
\(\operatorname{ob}\) vanishes on the image of
\(\Der_{\CC}(M,M)\).  For the derivations \(D_p\), one has
\[
 \deg(D_p)=\epsilon_p
 =\sigma-d+w_p.
\]
The degree estimate gives \(\epsilon_p\ge0\).  Since \(\ol h\) has
degree \(\sigma\) and
\[
 B_q=0\qquad(q\ge\sigma),
\]
it follows that
\[
 \operatorname{ob}(D_p)
 =D_p(\ol h)
 \in B_{\sigma+\epsilon_p}
 =0.
\]
The classes of the \(D_p\) span \(Q\).  Hence
\(\operatorname{ob}=0\) on all of \(\Der_{\CC}(M,B)\).  Every
derivation \(M\to B\) therefore defines a derivation of \(B\), and
\[
 \dim_{\CC}\Der_{\CC}(B,B)
 =
 \dim_{\CC}\Der_{\CC}(M,B).
\]
Together with the first step, this gives
\[
 \dim_{\CC}\Der_{\CC}(B,B)
 =
 \dim_{\CC}\Der_{\CC}(M,M).
\]

The AIM-generated note contained this preliminary two-step argument.
The authors checked every map, degree calculation, and exactness
statement, corrected the intermediate argument where necessary, and
wrote the final proof given in the manuscript.


\subsection{AIM-assisted examples}

AIM was used to explore and check examples that were later included in
the final section of the paper.  For the \(A_k\) singularities, the
calculation shows that the natural map
\[
 L(V)\longrightarrow L^*(V)
\]
has a one-dimensional kernel.  AIM-assisted calculations also gave
\[
 \lambda(V)=\lambda^*(V)=4
\]
for \(f=x^3+y^3+\sum z_i^2\), and a difference of one for
\(f=x^ay+y^2\).  For the family
\[
 f_s=x^3+y^3+sxy+z^2,
\]
 both dimensions are zero for \(s\ne0\), whereas for \(s=0\), both
dimensions are four.  Thus their common value may jump even though the
equality between them is preserved.

The authors selected the examples, verified all calculations, and
included them to illustrate the main result.

\subsection{Scope of assistance}

The main purpose of the paper is to prove
\[
 \lambda^*(V)=\lambda(V)
\]
for isolated hypersurface singularities with \(n\ge2\) and
\(\mult(f)\ge3\), thereby proving Conjecture~1.1.  The authors
formulated the problem and proposed comparing
\(\Der_{\CC}(M,M)\) and \(\Der_{\CC}(B,B)\) through
\(\Der_{\CC}(M,B)\).  A pre-manuscript note generated with AIM contained
a preliminary two-step argument for the weighted homogeneous case.  AIM was
also used for preliminary calculations of the examples.

The adjugate-Hessian derivations and the required degree estimate were
already known from earlier work and are not attributed to AIM.  The
authors checked and corrected the preliminary argument, supplied the
remaining arguments needed to prove the conjecture in full, and wrote
the final proof.  They take responsibility for every mathematical
statement and calculation in the paper.

\subsection{Internal development record}

This account is based on dated AIM session records from June 2026 and
an AIM workspace note dated 5 July 2026, both of which precede the
manuscript version dated 29 July 2026.  These materials record when the
exploratory calculations and the exact-sequence organization described
above were written down.  Because the records are not publicly
available, readers cannot use them to verify the origin of individual
arguments.

\section*{Acknowledgments} 
Yau is supported by the Tsinghua University Education Foundation.

 During the preparation of this manuscript, the AI Mathematician (AIM) agent~(see \cite{https://arxiv.org/html/2505.22451v1}) was used as a research assistance tool under the authors’ direction. It
was helpful for exploring certain technical aspects of commutative algebra
and for suggesting computational examples. All mathematical arguments,
computations, and statements in this paper were verified by
the authors, who take full responsibility for the correctness and content of
the manuscript.

\end{document}